\documentclass[11pt]{article}
\usepackage[dvips]{graphicx}
\usepackage{amsfonts}
\usepackage[english]{babel}
\usepackage[dvips]{color}
\usepackage{amsmath}
\usepackage{amsthm}
\usepackage{amssymb}
\usepackage{amsfonts}
\usepackage{xcolor}
\usepackage{caption}

\usepackage{amsthm}
\usepackage{xcolor}

\newtheorem{cor}{Corollary}[section]

\newtheorem{prop}{Proposition}[section]

\newtheorem{rem}{Remark}[section]

\numberwithin{equation}{section}

\begin{document}

\title{\Large {\bf
{On differentiation with respect to parameters of the functions of the Mittag-Leffler type}}}

\author{Sergei V. Rogosin$^{a,\ast}$,\footnote{Corresponding author}
Filippo Giraldi$^{b,c}$ Francesco Mainardi$^d$}

\maketitle

\smallskip
\begin{center}
{$^a$ Department of Economics, Belarusian State University,\\ Nezavisimosti ave 4,  220030 Minsk, Belarus,\\
$^b$ Section of Mathematics, International Telematic University Uninettuno, \\ Corso Vittorio Emanuele II 39, 00186, Rome, Italy\\
{$^c$ School of Chemistry and Physics, University of KwaZulu-Natal, \\ Westville Campus, Durban 4000, South Africa} \\
$^d$ Department of Physics $\&$ Astronomy,
University of Bologna, and INFN,
\\ Via Irnerio 46, I-40126 Bologna, Italy}
\end{center}

\begin{abstract}
The formal term-by-term differentiation with respect to parameters is demonstrated to be legitimate for the Mittag-Leffler type functions. The justification of differentiation formulas is made by using the concept of the uniform convergence. This approach is applied to the Mittag-Leffler function depending on two parameters and, additionally, for the $3$-parametric Mittag-Leffler functions (namely, for the Prabhakar function and the Le Roy type functions), as well as for the $4$-parametric Mittag-Leffler function (and, in particular, for the Wright function). The differentiation with respect to the involved parameters is discussed also in case those special functions which are represented via the Mellin-Barnes integrals.


\vspace{1mm}\noindent
differentiation,
 Mellin-Barnes integral
\\
{\bf Paper published on line 19 November 2024  in} 
\\ {\bf Integral Transforms and Special Functions},
\\
{\bf DOI:10.1080/10652469.2024.2428850}
\end{abstract}


\maketitle

\section{Introduction}
\label{Intro}

An interest to the differential properties of special functions with respect to their parameters is recently growing. Respective formulas become the sources for new classes of special functions as well as for differential equations of new types. They are important as for better understanding of the behaviour of special functions as for their applications (see, e.g. \cite{Lan00}). A number of articles are devoted to the differentiation of the Bessel-type functions, see e.g.
\cite{ApeKra85},  \cite{Bry16} and references therein. In particular,
the differential equation for the derivatives of the Bessel function with respect to parameter was used in \cite{Dun17} for derivation of
an integral representation for $\frac{\partial J_{\nu}(z)}{\partial \nu}$.

Recently,  Apelblat
\cite{Ape20} and   Apelblat and Mainardi  \cite{ApeMai22} proposed an approach for differentiation of the Mittag-Leffler and Wright  functions (see, e.g. \cite{GKMR}), respectively,   with respect to parameters.
The obtained formulas are derived by using the formal term-by-term differentiation of the series represented those special  functions.
 So, the derivatives with respect the parameters are again series
and one can ask where these series do converge.

Here we proposed two ways of justification of the above said formulas. For simplicity we apply our approaches to the Mittag-Leffler function
\begin{equation}
\label{M-L1}
E_{\alpha,\beta}(z) = \sum\limits_{k=0}^{\infty} \frac{z^k}{\Gamma(\alpha k + \beta)},
\end{equation}
which is highly related to the Wright function.  The same ideas can be realized for other special functions such as the Wright function, the Prabhakar function (three parametric Mittag-Leffler function) (see, e.g. \cite{GKMR}) {and Le Roy type functions recently extendedly studied in the series of articles \cite{RogDub23,RogDub24,P-KKiRoDu23a,P-KKiRoDu23b}}.

The first way of justifying the formal term-by-term differentiation is based on the uniform convergence property \cite{Har18,Rud76}. This property is described below for the sake of clarity. Let $f_0, f_1, \ldots$, be a sequence of real or complex valued functions defined, and differentiable, over an interval $\left[a, b\right]$, such that the following limit: $\lim_{n \to +\infty} f_n\left(x\right)$, exists and is finite for some element $x_0$ belonging to the interval $\left[a, b\right]$. Let the sequence of the derivatives $f^{\prime}_0, f^{\prime}_1, \ldots$, converges uniformly on the interval $\left[a, b\right]$. Under these conditions, the sequence $f_0, f_1, \ldots$ converges uniformly  to a function $f$ on the interval $\left[a, b\right]$, and $f^{\prime} \left(x\right)= \lim_{n\to +\infty}f^{\prime}_n \left(x\right)$, for every $x \in \left[a, b\right]$. The uniform convergence of the series composed by the term-by-term derivatives is demonstrated below by adopting the Weierstrass $M$-test theorem. The second way of justifying the term-by-term differentiation relies on the representation of the involved special functions via the Mellin-Barnes integrals.

\section{Differentiability of Mittag-Leffler function with respect to parameters}
\label{M-L}

It is known that for ${\mathrm{Re}}\, \alpha > 0, \beta\in {\mathbb C}$ the Mittag-Leffler function (\ref{M-L1}) is an entire function of the complex variable $z$. In this part we apply our approach only in the case $\alpha > 0, \beta\geq 0$ which is most relievable to application.
Term-by-term differentiation of series (\ref{M-L1}) gives
\begin{equation}
\label{M-L_diff1}
\frac{\partial E_{\alpha,\beta}(z)}{\partial \alpha} = - \sum\limits_{k=1}^{\infty} \frac{k \psi(\alpha k + \beta) z^k}{\Gamma(\alpha k + \beta)},
\end{equation}
where $\psi(x) = \frac{\Gamma^{\prime}(x)}{\Gamma(x)}$ is so-called $\psi$-function or the digamma function (see e.g. \cite[Ch. 5]{NIST}).

Let us fix $\alpha > 0, \beta\geq 0$ then $\alpha k + \beta > 0$ for all {$k\in {\mathbb N}$}.
To show convergence of series in (\ref{M-L_diff1}) we use the inequality
\begin{equation}
\label{psi1}
\ln (x + \frac{1}{2}) \leq \psi(x + 1) \leq \ln (x + e^{-\gamma}),
\end{equation}
valid for all $x > 0$ \cite{Alz97} with $\gamma := \lim\limits_{n\rightarrow \infty} \left(- \ln n + \sum\limits_{k=1}^{n} \frac{1}{k}\right)$ being the Euler-Mascheroni constant (see also \cite{GuoQi14} for the best possible inequalities of such type),
which leads to the following asymptotic formula
\begin{equation}
\label{psi2}
\psi(\alpha k + \beta) = {\ln \left(k [1 + o(1)]\right)},\;\;\; k \rightarrow \infty,
\end{equation}
and also Stirling-like formula (see, e.g. \cite[Appendix A]{GKMR})
\begin{equation}
\label{gamma1}
\Gamma(\alpha k + \beta) = \sqrt{2\pi} (\alpha k)^{\alpha k + \beta - 1/2} e^{-\alpha k} [1 + o(1)],\;\;\; k \rightarrow \infty,
\end{equation}
which leads to the following asymptotic formula \cite[1.18(4)]{BatErd54a}
\begin{equation}
\label{gamma2}
\frac{\Gamma(\alpha (k+1) + \beta)}{\Gamma(\alpha k + \beta)} = (\alpha k)^{\alpha} \left[1 + \frac{\alpha(\alpha - 1)}{2 \alpha k} + O\left(\frac{1}{(\alpha k)^2}\right)\right],
k \rightarrow \infty.
\end{equation}

Thus, we can calculate the radius of convergence of the series in (\ref{M-L_diff1}), which has the form $\sum\limits_{k=1}^{\infty} a_k z^k$,
\begin{equation}
\label{radius1}
R = \lim_{k\rightarrow\infty} \frac{\mid a_k\mid}{\mid a_{k+1}\mid} = \lim_{k\rightarrow\infty} \frac{k \psi(\alpha k + \beta)}{(k+1) \psi(\alpha (k+1) + \beta)} \frac{\Gamma(\alpha (k + 1) + \beta)}{\Gamma(\alpha k + \beta)} =
\end{equation}
$$
\lim_{k\rightarrow\infty} \frac{k (\ln k [1 + o(1)])}{(k + 1) (\ln (k + 1) [1 + o(1)])} (\alpha k)^{\alpha} \left[1 + \frac{\alpha(\alpha - 1)}{2 \alpha k} + O\left(\frac{1}{(\alpha k)^2}\right)\right] = \infty.
$$
Therefore, the series in (\ref{M-L_diff1}) is converging for all $z\in {\mathbb C}$ for any fixed pair $(\alpha, \beta), \alpha > 0, \beta\geq 0$. It gives existence of the derivative $\frac{\partial E_{\alpha,\beta}(z)}{\partial \alpha}$ for all $\alpha > 0, \beta\geq 0$.

\begin{prop}
\label{Diff_ML}
The Mittag-Leffler function $E_{\alpha,\beta}(z)$ is differentiable as a function of the variable $\alpha > 0$ for any fixed value of parameter $\beta \geq 0$ and $z\in {\mathbb C}$ and formula (\ref{M-L_diff1}) holds valid.
\end{prop}
\begin{proof}
The Mittag-Leffler function $E_{\alpha,\beta}(z)$ is differentiable as a function of the variable $\alpha$ whenever the series in (\ref{M-L_diff1}) is converging uniformly with respect to parameter $\alpha\in [a, b]$ and $\beta \in [0, B]$, where $a > 0, b > a, B > 0$ are arbitrary fixed numbers.
It is known (see e.g. \cite{GuoQi14}) that for all $x \geq 1$ the following inequality hold
\begin{equation}
\label{psi3}
\ln (x + \frac{1}{2}) \leq \psi(x + 1) \leq \ln (x + e^{1-\gamma} - 1).
\end{equation}
For all  $\alpha\in [a, b]$ and $\beta \in [0, B]$ we take $k_1 = \frac{2}{a}$ (to have argument of psi-function in (\ref{psi3}) greater than 1 for all $x = \alpha k + \beta$ with $\alpha\in [a, b]$ and $\beta \in [0, B]$) and $k_2 = \frac{3 - B - e^{1 - \gamma}}{b}$ (to the positive value of logarithm in the right hand-side of (\ref{psi4}) below).
Then for all $k > k_0 = \max\{k_1, k_2\}$ and all  $\alpha\in [a, b]$ and $\beta \in [0, B]$ we have
$$
\alpha k + \beta > a k + \beta > a k > 1,
$$
and thus it follows from (\ref{psi3}) that
\begin{equation}
\label{psi4}
\psi(\alpha k + \beta) \leq \ln (\alpha k + \beta + e^{1-\gamma} - 2) \leq \ln (b k + B + e^{1-\gamma} - 2),
\end{equation}
where $b k + B + e^{1-\gamma} - 2 > 1$ since $k > k_2$.

It follows from the inequality
\begin{equation}
\label{uniform1}
\left| a_k z^k\right| = \frac{|k \psi(\alpha k + \beta)|}{|\Gamma(\alpha k + \beta)|} |z|^k \leq \frac{|k \ln (b k + B + e^{1-\gamma} - 2)|}{|\Gamma(a k)|} |z|^k
\end{equation}
valid for for all $k \geq k_0$, and convergence of the positive series
\begin{equation}
\label{uniform2}
\sum\limits_{k=k_0}^{\infty} \frac{k \ln (b k + B + e^{1-\gamma} - 2)}{\Gamma(a k + \beta)} |z|^k \textcolor[rgb]{0.98,0.00,0.00}{\leq} \sum\limits_{k=k_0}^{\infty} \frac{\ln (b k + B + e^{1-\gamma} - 2)}{\ln (a k)} \frac{k \ln (a k)}{\Gamma(a k)} |z|^k.
\end{equation}
The latter series is converging since it follows from previous consideration that
the series
$$
\sum\limits_{k=k_0}^{\infty} \frac{k \ln (a k)}{\Gamma(a k)} |z|^k
$$
converges and the sequence
$$
\frac{\ln (b k + B + e^{1-\gamma} - 2)}{\ln (a k)}
$$
is monotone, positive and bounded for all $k > k_0$.
\end{proof}

In a similar way one can justify the differential formula for differentiation of the Mittag-Leffler function $E_{\alpha,\beta}(z)$ with respect to the second parameter $\beta > 0$:
\begin{equation}
\label{M-L_diff2}
\frac{\partial E_{\alpha,\beta}(z)}{\partial \beta} = - \sum\limits_{k=0}^{\infty} \frac{\psi(\alpha k + \beta) z^k}{\Gamma(\alpha k + \beta)}.
\end{equation}
Slight modification of the above approach allow us to show that formula (\ref{M-L_diff2}) is also valid for all complex values of the parameter $\beta$.
\begin{prop}
\label{Diff_ML}
The Mittag-Leffler function $E_{\alpha,\beta}(z)$ is differentiable (analytic) as a function of the variable $\beta\in {\mathbb C}$ for any fixed value of parameters ${\mathrm{Re}}\, \alpha > 0$, $z\in {\mathbb C}$ and formula  (\ref{M-L_diff2})
holds valid.
\end{prop}
\begin{proof}
In the case of complex values of $\beta$ it is impossible to use inequalities for psi-function of positive argument as (\ref{psi1}), (\ref{psi3}). Another constrain is that psi-function has zeroes (one positive belonging to the interval $[1, 2]$ and countable set of negative zeroes, see e.g. \cite{MezHof17}). Anyway, the proof of the Proposition can be obtained by using asymptotic formula
\begin{equation}
\label{psi5}
\psi(z) \sim \ln z - \frac{1}{z} \;\;\; \Leftrightarrow  \;\;\; \psi(z) = \left(\ln z - \frac{1}{z}\right) \left[1 + \delta(z)\right]
\end{equation}
for complex numbers $z$ with large modulus ($|z|\rightarrow \infty$) in the sector $|\arg z| < \pi - \varepsilon$ with some infinitesimally small positive constant $\varepsilon > 0$, and $\delta(z) = o(1)$ as $|z|\rightarrow \infty$. Note that for all fixed $\alpha > 0$, $\beta\in {\mathbb C}$ there exists a sufficiently large $k_1$ such that for $k > k_1$ the real part of the expression $\alpha k + \beta$ is greater than 2, i.e. ${\mathrm{Re}}\, (\alpha k + \beta) > 2, k > k_1$. Thus for all fixed $\alpha > 0$, $\beta\in {\mathbb C}$ and all  $k > k_1$ the function $\psi(\alpha k + \beta)$ has no zeroes and $\Gamma(\alpha k + \beta)$ has neither zeroes no poles.

Using asymptotic relations (\ref{psi5}) and (\ref{gamma2}) we can determine the radius $R$ of convergence of the series (\ref{M-L_diff2}), which has the form $\sum\limits_{k=1}^{\infty} b_k z^k$,
\begin{equation}
\label{radius2}
R = \lim_{k\rightarrow\infty} \frac{\mid b_k\mid}{\mid b_{k+1}\mid} = \lim_{k\rightarrow\infty} \frac{\mid\psi(\alpha k + \beta)\mid}{\mid\psi(\alpha (k+1) + \beta)\mid} \frac{\mid\Gamma(\alpha (k + 1) + \beta)\mid}{\mid\Gamma(\alpha k + \beta)\mid} =
\end{equation}
$$
\lim_{k\rightarrow\infty} \frac{\left| \left(\ln (\alpha k + \beta) - \frac{1}{\alpha k + \beta}\right) \left[1 + \delta(\alpha k + \beta)\right] \right|}{\left| \left(\ln (\alpha (k+1) + \beta) - \frac{1}{\alpha (k+1) + \beta}\right) \left[1 + \delta(\alpha k + \beta)\right]\right|} \times
$$
$$
|(\alpha k)^{\alpha}| \left| \left[1 + \frac{\alpha(\alpha - 1)}{2 \alpha k} + O\left(\frac{1}{(\alpha k)^2}\right)\right]\right | = \infty.
$$

Let us now  prove the differentiability (the analyticity) of the Mittag-Leffler function $E_{\alpha,\beta}(z)$ with respect to the complex  parameter $\beta$.  We take the domain of parameters in the form $$D := \left\{(\alpha, \beta)\in {\mathbb C}^2: {\mathrm{Re}}\, \alpha\in [a,b], {\mathrm{Im}}\, \alpha \leq A; |\beta| \leq B\right\},$$ where $a, b, A, B$ are arbitrary fixed positive numbers, $b > a$.

The coefficients of the series (\ref{M-L_diff2}) are equal
$$
b_k = - \frac{\psi(\alpha k + \beta)}{\Gamma(\alpha k + \beta)}.
$$
We get uniform estimate of the modulus $|b_k|$ for sufficiently large $k$ in the above described domain $D$ of parameters.

In the asymptotic formula (\ref{psi4}) the term $\delta(z)$ is small enough for sufficiently large $|z|$. It gives, in particular, that $\exists r_2 > 0$ such that $\forall z, |z| > r_2 \Rightarrow |\delta(z)| < 1$. If $(\alpha, \beta)\in D$  then
\begin{equation}
\label{estimate1}
|\alpha k + \beta| \geq |\alpha| k - |\beta| \geq \sqrt{{\mathrm{Re}}^2 \alpha + {\mathrm{Im}}^2 \alpha} k - |\beta| \geq |{\mathrm{Re}} \alpha| k - |\beta| \geq a k - B
\end{equation}
Taking $k_2 = \frac{B + r_2}{a}$ we obtain that for all $k > k_2$ and all $(\alpha, \beta)\in D$ we have $|\delta(\alpha k + \beta)| < 1$. Hence for all $k > k_2$ and all $(\alpha, \beta)\in D$ it follows
\begin{equation}
\label{psi5}
|\psi(\alpha k + \beta)| \leq \left[|\ln (\alpha k + \beta)| + \frac{1}{|\alpha k + \beta|}\right] (1 + |\delta(\alpha k + \beta)|) \leq
\end{equation}
$$
2 \left(\sqrt{\ln^2 (\sqrt{b^2 + A^2} k + B) + \pi^2} +\frac{1}{a k - B}\right).
$$

From the other side it follows from (\ref{gamma1}) that
$$
\Gamma(\alpha k + \beta) = \sqrt{2\pi} (\alpha k)^{\alpha k + \beta - 1/2} e^{-\alpha k} [1 + \omega(\alpha k + \beta)],
$$
where $\omega(z)$ is sufficiently small for large enough $|z|$.  It gives, in particular, that there exists $r_3 > 0$ such that $\forall z, |z| > r_3 \Rightarrow |\omega(z)| < 1/2$. Then we can find $k_3 = \frac{B + r_3}{a}$. Thus for all $k > k_3$ and all $(\alpha, \beta)\in D$ we have $|\omega(\alpha k + \beta)| < 1/2$. Hence for all $k > k_3$ and all $(\alpha, \beta)\in D$ it follows
\begin{equation}
\label{gamma3}
|\Gamma(\alpha k + \beta)| = \sqrt{2\pi} \left|(\alpha k)^{\alpha k + \beta - 1/2}\right| \left|e^{-\alpha k}\right| \left|[1 + \omega(\alpha k + \beta)]\right| \geq
\end{equation}
$$
 \sqrt{\pi/2} (a k)^{ak - B -1/2} e^{- \pi (A + B)} \cdot e^{- b k}.
$$
The estimates (\ref{psi5}), (\ref{gamma3}) give that for all $k > k_0 = \max \{k_1, k_2, k_3 \}$ and for all $(\alpha, \beta)\in D$ and any fixed $z$ the following inequality holds
\begin{equation}
\label{estimate_fin}
\left|b_k z^k\right| \leq \frac{2 \left(\sqrt{\ln^2 (\sqrt{b^2 + A^2} k + B) + \pi^2} +\frac{1}{a k - B}\right)}{\sqrt{\pi/2} (a k)^{ak - B -1/2} e^{- \pi (A + B)} \cdot e^{- b k}} |z|^k.
\end{equation}
Therefore the uniform convergence of the series $\sum\limits_{k=1}^{\infty} b_k z^k$ holds true since the positive series
$$
\sum\limits_{k=k_0}^{\infty} \frac{2 \left(\sqrt{\ln^2 (\sqrt{b^2 + A^2} k + B) + \pi^2} +\frac{1}{a k - B}\right)}{\sqrt{\pi/2} (a k)^{ak - B -1/2} e^{- \pi (A + B)} \cdot e^{- b k}} |z|^k
$$
converges. The later follows from the fact that the expression in the numerator behaves as $C_1 \ln k$, and that in the denominator behaves as $C_2 k^{a k}$, where $C_1, C_2$ are certain constants.
Since the constants $a, b, A, B$ are chosen arbitrary, it gives differentiability (analyticity) of the function $E_{\alpha,\beta}(z)$ with respect to $\beta\in {\mathbb C}$ for arbitrary $\alpha, {\mathrm{Re}}\, \alpha > 0$.
\end{proof}
In a similar way we can derive the differentiability results for functions of the Mittag-Leffler type depending on several parameters.
Let us formulate this result in the case 4-parametric Mittag-Leffler function (see \cite{RogKor10}, \cite[Sec. 6.1]{GKMR})
\begin{equation}
\label{ML4par}
E_{\alpha_1,\beta_1;\alpha_2,\beta_2}(z) = \sum\limits_{k=0}^{\infty} \frac{z^k}{\Gamma(\alpha_1 k + \beta_1) \Gamma(\alpha_2 k + \beta_2)},
\end{equation}
which is defined and analytic for all $z\in {\mathbb C}$ whenever ${\mathrm{Re}}\, (\alpha_1 + \alpha_2) > 0; \beta_1\in {\mathbb C}, \beta_2\in {\mathbb C}$.
\begin{cor}
\label{4par}
The 4-parametric Mittag-Leffler function (\ref{ML4par}) is differentiable with respect to parameters $\alpha_1,\beta_1;\alpha_2,\beta_2$ in the domain
$$D_4 := \left\{(\alpha_1,\beta_1,\alpha_2\beta_2): {\mathrm{Re}}\, (\alpha_1 + \alpha_2) > 0; \beta_1\in {\mathbb C}, \beta_2\in {\mathbb C}\right\}$$
and the following formulas are satisfied
\begin{equation}
\label{ML4par_diff1}
\frac{\partial E_{\alpha_1,\beta_1;\alpha_2\beta_2}(z)}{\partial \alpha_j} = - \sum\limits_{k=1}^{\infty} \frac{k \psi(\alpha_j k + \beta_j)}{\Gamma(\alpha_1 k + \beta_1) \Gamma(\alpha_2 k + \beta_2)} z^k, \;\;\; j = 1, 2,
\end{equation}
\begin{equation}
\label{ML4par_diff2}
\frac{\partial E_{\alpha_1,\beta_1;\alpha_2,\beta_2}(z)}{\partial \beta_j} = - \sum\limits_{k=0}^{\infty} \frac{\psi(\alpha_j k + \beta_j)}{\Gamma(\alpha_1 k + \beta_1) \Gamma(\alpha_2 k + \beta_2)} z^k, \;\;\; j = 1, 2.
\end{equation}
\end{cor}
The Wright function (see \cite[Ch. 7]{GKMR}, \cite{Luc19})
\begin{equation}
\label{Wright}
W_{\alpha,\beta}(z) = \sum\limits_{k=0}^{\infty} \frac{z^k}{k! \Gamma(\alpha k + \beta)} = \sum\limits_{k=0}^{\infty} \frac{z^k}{\Gamma(k + 1) \Gamma(\alpha k + \beta)}
\end{equation}
can be considered as a special case of the 4-parametric Mittag-Leffler function $W_{\alpha,\beta}(z) = E_{\alpha,\beta;1,1}(z)$. The corresponding result has the form
\begin{cor}
\label{Wright1}
The Wright function (\ref{Wright}) is differentiable with respect to parameters $\alpha,\beta$ for all $\alpha, {\mathrm{Re}}\, \alpha > - 1$, $\beta\in {\mathbb C}$
and the following formulas are satisfied
\begin{equation}
\label{Wright_diff1}
\frac{\partial W_{\alpha,\beta}(z)}{\partial \alpha} = - \sum\limits_{k=1}^{\infty} \frac{k \psi(\alpha k + \beta)}{k! \Gamma(\alpha k + \beta)} z^k,
\end{equation}
\begin{equation}
\label{Wright_diff2}
\frac{\partial W_{\alpha,\beta}(z)}{\partial \beta} = - \sum\limits_{k=0}^{\infty} \frac{\psi(\alpha k + \beta)}{k! \Gamma(\alpha k + \beta)} z^k.
\end{equation}
\end{cor}

In a similar way we can obtain and justify the formulas for derivatives with respect to parameters of the three parametric Mittag-Leffler function (or Prabhakar function) (see \cite{GarGar18}, \cite[Ch. 5]{GKMR})
\begin{equation}
\label{ML3par}
E_{\alpha,\beta}^{\gamma}(z) = \sum\limits_{k=0}^{\infty} \frac{(\gamma)_k z^k}{k! \Gamma(\alpha k + \beta)} = \sum\limits_{k=0}^{\infty} \frac{\Gamma(\gamma + k) z^k}{\Gamma(\gamma) k! \Gamma(\alpha k + \beta)},
\end{equation}
which is defined and analytic for all $z\in {\mathbb C}$ whenever ${\mathrm{Re}}\, (\alpha) > 0; \beta\in {\mathbb C}, \gamma\in {\mathbb C}, \gamma \not= -1, -2, -3, \ldots$.
A little bit more cumbersome is the differentiation with respect to the parameter $\gamma$ since
$$
\frac{\partial}{\partial \gamma} \left[\frac{\Gamma(\gamma + k)}{\Gamma(\gamma)}\right] = \frac{\Gamma^{\prime}(\gamma + k) \Gamma(\gamma) - \Gamma^{\prime}(\gamma) \Gamma(\gamma + k)}{[\Gamma(\gamma)]^2} =
$$
$$
\left[\frac{\Gamma(\gamma + k)}{\Gamma(\gamma)}\right] \left\{\psi(\gamma + k) - \psi(\gamma)\right\}.
$$
\begin{cor}
\label{3par}
The 3-parametric Mittag-Leffler (Prabhakar) function (\ref{ML3par}) is differentiable with respect to parameters $\alpha,\beta,\gamma$ for all $\alpha, {\mathrm{Re}}\, \alpha > 0$, $\beta,\gamma\in {\mathbb C}, \gamma \not= -1, -2, -3, \ldots$
and the following formulas are satisfied
\begin{equation}
\label{ML3par_diff1}
\frac{\partial E_{\alpha,\beta}^{\gamma}(z)}{\partial \alpha} = - \sum\limits_{k=1}^{\infty} \frac{k \psi(\alpha k + \beta) \Gamma(\gamma + k)}{k! \Gamma(\gamma) \Gamma(\alpha k + \beta)} z^k,
\end{equation}
\begin{equation}
\label{ML3par_diff2}
\frac{\partial E_{\alpha,\beta}^{\gamma}(z)}{\partial \beta} = - \sum\limits_{k=0}^{\infty} \frac{\psi(\alpha k + \beta) \Gamma(\gamma + k)}{k! \Gamma(\gamma) \Gamma(\alpha k + \beta)} z^k,
\end{equation}
\begin{equation}
\label{ML3par_diff3}
\frac{\partial E_{\alpha,\beta}^{\gamma}(z)}{\partial \gamma} = \sum\limits_{k=0}^{\infty} \frac{\psi(\gamma + k) \Gamma(\gamma + k)}{k! \Gamma(\gamma) \Gamma(\alpha k + \beta)} z^k - \psi(\gamma) \sum\limits_{k=0}^{\infty} \frac{\Gamma(\gamma + k) z^k}{\Gamma(\gamma) k! \Gamma(\alpha k + \beta)}.
\end{equation}
\end{cor}
Both series in the right-hand side of (\ref{ML3par_diff3}) converge uniformly in the corresponding domains. The first one can be treated in the way similar to that  in Proposition \ref{Diff_ML}, but the second series coincides up to the constant multiplier with the initial Prabhakar function.

One more non-standard way of the differentiation is that for the so called Le Roy type function
\begin{equation}
\label{LeRoy1}
F_{\alpha,\beta}^{(\gamma)}(z) = \sum\limits_{k=0}^{\infty}  \frac{z^k}{[\Gamma(\alpha k + \beta)]^{\gamma}},
\end{equation}
which is known (see, e.g. \cite{GaRoMa17}, \cite[Sec. 5.3]{GKMR}) to be an entire function of variable $z$ whenever $\alpha, {\mathrm{Re}}\, \alpha > 0$, $\beta\in {\mathbb C}$, $\gamma > 0$.

The main novelty gives the derivative with respect to the third parameter $\gamma$ since
$$
\frac{\partial}{\partial \gamma} \left[\Gamma(\alpha k + \beta)\right]^{-\gamma} =
- \left[\Gamma(\alpha k + \beta)\right]^{-\gamma} \ln\, (\Gamma(\alpha k + \beta))  =
- \frac{\ln\, (\Gamma(\alpha k + \beta))}{\left[\Gamma(\alpha k + \beta)\right]^{\gamma}}.
$$
Anyway the series with this kind of derivative of the coefficient can be handled by using the same argument as before since numerator of this fraction
behaves as $C_1 k \ln\, k$ as $k \rightarrow \infty$, but denominator behaves as $C_2 k^{({\mathrm{Re}}\, \alpha) \gamma k}$  as $k \rightarrow \infty$, where $C_1, C_2$ are certain constants (see, e.g. \cite[Sec. 5.3]{GKMR}). Therefore the following results hold
\begin{cor}
\label{3par}
The Le Roy type function (\ref{LeRoy1}) is differentiable with respect to parameters $\alpha,\beta,\gamma$ for all $\alpha, {\mathrm{Re}}\, \alpha > 0$, $\beta\in {\mathbb C}, \gamma > 0$
and the following formulas are satisfied
\begin{equation}
\label{Le Roy_diff1}
\frac{\partial F_{\alpha,\beta}^{(\gamma)}(z)}{\partial \alpha} = - \gamma \sum\limits_{k=1}^{\infty} \frac{k \psi(\alpha k + \beta)}
{\left[\Gamma(\alpha k + \beta)\right]^{\gamma}} z^k,
\end{equation}
\begin{equation}
\label{Le Roy_diff2}
\frac{\partial F_{\alpha,\beta}^{(\gamma)}(z)}{\partial \beta} = - \gamma \sum\limits_{k=1}^{\infty} \frac{\psi(\alpha k + \beta)}
{\left[\Gamma(\alpha k + \beta)\right]^{\gamma}} z^k,
\end{equation}
\begin{equation}
\label{Le Roy_diff3}
\frac{\partial F_{\alpha,\beta}^{(\gamma)}(z)}{\partial \gamma} = - \sum\limits_{k=1}^{\infty} \frac{\ln\, \Gamma(\alpha k + \beta)}
{\left[\Gamma(\alpha k + \beta)\right]^{\gamma}} z^k.
\end{equation}
\end{cor}

\section{Differentiation of functions represented via the Mellin-Barnes integrals}

We can also calculate the derivative of the special functions which possess the Mellin-Barnes representation. For the Mittag-Leffler function $E_{\alpha,\beta}(z)$
with $\alpha > 0, \beta\geq 0$ such representation has the form
\begin{equation}
 \label{Mellin-Barnes1}
E_{\alpha,\beta}(z) = \frac{1}{2 \pi i} \int\limits_{{\mathcal L}_{-\infty}} \frac{\Gamma(s) \Gamma(1 - s)}{\Gamma(\beta - \alpha s)} (- z)^{-s} ds, \; |\arg\, z| < \pi,
\end{equation}
where the contour of integration ${\mathcal L}_{-\infty}$ is a line which starts at $-\infty - i \varphi$ and ends at $\infty + i \varphi$ with a fixed sufficiently small $\varphi > 0$, crossing the real line at a point $c, 0 < c < 1$. Thus the contour ${\mathcal L}_{-\infty}$ leaves all poles $s = 0, - 1, -2, \ldots $ of $\Gamma(s)$ to the left and all poles $s = 1, 2, 3, \ldots $ of $\Gamma(1 - s)$ to the right. Here the branch of the multi-valued function $(-z)^{-s}$ is defined in the
complex plane cut along negative semi-axis and
$$
(-z)^{-s} = \exp \{-s[\log |z| + i\arg (-z)]\},
$$
where $\arg (-z)$ is any arbitrary chosen branch of ${\mathrm{Arg}} (-z)$ (e.g. the principal branch of it).

Formal derivatives of this integral with respect to parameters give the following results
\begin{equation}
\label{MB_diff1}
\frac{\partial E_{\alpha,\beta}(z)}{\partial \alpha} = \frac{1}{2 \pi i} \int\limits_{{\mathcal L}_{-\infty}} \frac{s \psi(\beta - \alpha s) \Gamma(s) \Gamma(1 - s)}{\Gamma(\beta - \alpha s)} (- z)^{-s} ds, \; |\arg\, z| < \pi,
\end{equation}
\begin{equation}
\label{MB_diff2}
\frac{\partial E_{\alpha,\beta}(z)}{\partial \beta} = - \frac{1}{2 \pi i} \int\limits_{{\mathcal L}_{-\infty}} \frac{\psi(\beta - \alpha s) \Gamma(s) \Gamma(1 - s)}{\Gamma(\beta - \alpha s)} (- z)^{-s} ds, \; |\arg\, z| < \pi.
\end{equation}
In a sense these formulas are similar to the Mellin-Barnes formulas. But they contain in the integrands not only ratio of the Gamma-function but also the derivative of Gamma-function.
The advantage of formulas (\ref{MB_diff1}), (\ref{MB_diff2}) is in possibility to get asymptotic results following the technique developed in
\cite{ParKam01}.

\begin{prop}
\label{ML-MB1}
Let the Mittag-Leffler function $E_{\alpha,\beta}(z)$ with $\alpha > 0, \beta\geq 0$ be represented by the Mellin-Barnes integral (\ref{Mellin-Barnes1}). Then it can be differentiable with respect to parameters and the formulas (\ref{MB_diff1}), (\ref{MB_diff2}) for derivatives
$\frac{\partial E_{\alpha,\beta}(z)}{\partial \alpha}$, $\frac{\partial E_{\alpha,\beta}(z)}{\partial \beta}$ are valid for all $\alpha > 0, \beta\geq 0$.
\end{prop}
\begin{proof}
First we show the convergence of the integral in (\ref{MB_diff1}) for each fixed pair $\alpha,\beta$, $\alpha > 0, \beta\geq 0$.

We note that the ratio $\frac{ s \psi(\beta - \alpha s)}{\Gamma(\beta - \alpha s)}$ has no pole in the complex plane, and only poles of the integrand are those of $\Gamma(s) \Gamma(1-s)$. The later do not lie on the contour of integration (they are separated by the line ${\mathcal L}_{-\infty}$). So, to show  the convergence of the integral (\ref{MB_diff1}) it suffices to get a proper asymptotic of the integrand.

The psi-function $\psi(z)$ has the following asymptotics
$$
\psi(z) \sim \ln\, z - \frac{1}{2 z}, \; as\; |z| \rightarrow \infty, |\arg z| < \pi - \varepsilon.
$$
Hence on the line $s = - x \pm i \varphi$ we have that the following asymptotic formula holds
$$
|\psi(\beta - \alpha(- x \pm i \varphi))| = \ln\, |\alpha x| \left(1 + O\left(\frac{1}{x}\right)\right), \; as\; x \rightarrow +\infty.
$$
It is known that reflection formula for the Gamma-function reads
$$
\Gamma(z) \Gamma(1 - z) = \frac{\pi}{\sin \pi z}, \;\;\; z\not\in {\mathbb Z}.
$$
Thus for $z = - x \pm i \varphi, x > 0,$ we have
$$
\Gamma(- x \pm i \varphi) \Gamma(1 + x \mp i \varphi) = \frac{2\pi i}{\sin \pi (- x \pm i \varphi)} =
$$
$$\frac{- 2 \pi i}{\cos (\pi x) [e^{\mp \pi \varphi} - e^{\pm \pi \varphi}] - i \sin (\pi  x) [e^{\mp \pi \varphi} + e^{\pm \pi \varphi}]}
$$
and
$$
|\Gamma(- x \pm i \varphi) \Gamma(1 + x \mp i \varphi)| = \frac{\pi}{\sqrt{\sinh^2 \pi \varphi + \sin^2 \pi x}} < \frac{\pi}{\sinh \pi \varphi}.
$$
Therefore this product is bounded on the line ${\mathcal L}_{- \infty}$.

Modulus of the function $(-z)^{-s}$ on the curve  ${\mathcal L}_{- \infty}$ (i.e. for $s = - x \pm i \varphi$) is equal
$$
\left|(-z)^{-s}\right| = |z|^{x} e^{\pm \varphi \arg (- z)}
$$

At last, from the Stirling asymptotic formula it follows (see, e.g. \cite[(1.5.12)]{KiSrTr06})
\begin{equation}
\label{gamma4}
|\Gamma(\beta + \alpha x \mp i \alpha\varphi| = \sqrt{2\pi} (\alpha x)^{\alpha x - 1/2} e^{- \alpha x - \beta - \pi [1 \mp \varphi}]/2 \left[1 + O\left(\frac{1}{x}\right)\right], \;\;\; x  \rightarrow +\infty.
\end{equation}
Combining the above estimates we have that up to the constant the integral
$$
\int\limits_{{\mathcal L}_{-\infty}} \left|\frac{s \psi(\beta - \alpha s) \Gamma(s) \Gamma(1 - s)}{\Gamma(\beta - \alpha s)} (- z)^{-s}\right| \left|ds\right|
$$
is bounded by the sum of two positive integrals (related to different sign behind $\varphi$)
\begin{equation}
\label{MB_estimate1}
\int\limits_{0}^{+\infty} \frac{(\alpha x)(\ln\, \alpha x)  \frac{\pi}{\sinh \pi \varphi}}{(\alpha x)^{\alpha x - 1/2} e^{- \alpha x - \beta - \pi [1 \mp \varphi]/2}} |z|^{x} e^{\pm \varphi \arg (- z)} dx,
\end{equation}
which are converging since the term in the denominator $(\alpha x)^{\alpha x}$ essentially dominates all other terms.

To show that this convergence is uniform with respect to parameters we fix arbitrary positive numbers $a, b, B, R$ and suppose that
$\alpha\in [a, b], \beta\in [0, B], |z| \leq R$. Then we can estimate the expression under the integral sign in (\ref{MB_estimate1})
\begin{equation}
\label{MB_estimate2}
\frac{(\alpha x)(\ln\, \alpha x)  \frac{\pi}{\sinh \pi \varphi}}{(\alpha x)^{\alpha x - 1/2} e^{- \alpha x - \beta - \pi [1 \mp \varphi]/2}} |z|^{x} e^{\pm \varphi \arg (- z)} \leq \frac{(b x)(\ln\, b x)  \frac{\pi}{\sinh \pi \varphi}}{(a x)^{a x - 1/2} e^{- b x - B - \pi [1 \mp \varphi]/2}} |R|^{x} e^{\varphi \pi}.
\end{equation}
Both integrals of such functions exists and do not depend on parameters. This gives differentiability of the Mittag-Leffler function $E_{\alpha,\beta}(z)$ with respect to the parameter $\alpha$.

In a similar way we can prove the differentiability of $E_{\alpha,\beta}(z)$ with respect to the parameter $\beta$.
\end{proof}

\begin{rem}
 \label{MB_other}
The above approach can be applied to show differentiability  with respect to parameters other function possessing representation via the Mellin-Barnes integral, namely

\vspace{3mm}
- the Prabhakar function (three parametric Mittag-Leffler function)
\begin{equation}
 \label{Prabhakar1}
 E_{\alpha,\beta}^{\gamma}(z) = \sum\limits_{k=0}^{\infty} \frac{(\gamma)_k}{k! \Gamma(\alpha k + \beta)} z^k,
 \end{equation}
 which possesses the Mellin-Barnes-type representation of the following form
 \begin{equation}
 \label{Prabhakar2}
 E_{\alpha,\beta}^{\gamma}(z) = \frac{1}{\Gamma(\gamma)} \frac{1}{2 \pi i} \int\limits_{{\mathcal L}} \frac{\Gamma(s) \Gamma(\gamma - s)}{\Gamma(\beta - \alpha s)} (- z)^{-s} ds, \; |\arg\, z| < \pi,
\end{equation}
with properly chosen contour ${\mathcal L}$;

\vspace{3mm}
- the four-parametric Mittag-Leffler function
\begin{equation}
 \label{Prabhakar1}
 E_{\alpha_1,\beta_1; \alpha_2,\beta_2}(z) = \sum\limits_{k=0}^{\infty} \frac{z^k}{\Gamma(\alpha_1 k + \beta_1) \Gamma(\alpha_2 k + \beta_2)},
 \end{equation}
 which possesses the Mellin-Barnes-type representation of the following form
 \begin{equation}
 \label{Prabhakar2}
 E_{\alpha_1,\beta_1; \alpha_2,\beta_2}(z) =  \frac{1}{2 \pi i} \int\limits_{{\mathcal L}} \frac{\Gamma(s) \Gamma(1 - s)}{\Gamma(\beta_1 - \alpha_1 s)
 \Gamma(\beta_2 - \alpha_2 s)} (- z)^{-s} ds, \; |\arg\, z| < \pi,
\end{equation}
with properly chosen contour ${\mathcal L}$;

\vspace{3mm}
- the Le Roy-type function
\begin{equation}
 \label{LeRoy2}
 F_{\alpha,\beta}^{(\gamma)}(z) = \sum\limits_{k=0}^{\infty} \frac{z^k}{\left[\Gamma(a k + \beta)\right]^{\gamma}},
 \end{equation}
 which possesses the Mellin-Barnes-type representation of the following form
 \begin{equation}
 \label{LeRoy3}
 F_{\alpha,\beta}^{(\gamma)}(z) =  \frac{1}{2 \pi i} \int\limits_{{\mathcal L}} \frac{\Gamma(- s) \Gamma(1 + s)}{\left[\Gamma(\beta + \alpha s)\right]^{\gamma}} (-z)^{s} ds + \frac{1}{[\Gamma(\beta)]^{\gamma}}, \; |\arg\, z| < \pi,
\end{equation}
with properly chosen contour ${\mathcal L}$.
\end{rem}
\newpage
{\section{Conclusions}
The legitimacy of the derivation with respect to parameters of Mittag-Leffler and Wright functions was easy to show just with a simple equality involving the Gamma function. 
\\
We also obtain such derivatives using the Mellin-Barnes integral representation of the involving special functions.
\\
It should be noted that analogous results can be obtained by using another type of integral representations with integration along Schl\"afli-type contour initially introduced in the study of Bessel type functions. Thus in \cite{AskAns24} the differentiations of the three-parameter Mittag-Leffler functions with respect to parameters are introduced and the steepest descent method is applied for modification of the asymptotic expansions for their large parameters. Schl\"afli-type integral representations for Kelvin functions  (see \cite[Sec. 10.3]{NIST}) are introduced and discussed in \cite{Ape91}.}
\\
{Special functions are used for the descriptions of various physical phenomena. Recently, special functions of Mittag-Leffler and Wright type have been adopted in quantum physics to generalize coherent states. In this regard, see e.g. \cite{SiPeSo99,GaGiMa19,GirMai23,Gir23a,Gir23b}.}
\\
{The second motivation might be the search for PDEs with respect to the parameters which are solved by these special functions.
This study is not yet finalized and will be continued in our further publications.}

\section*{Acknowledgments}
The  work by SR was supported by the the State Program of the Scientific Investigations ``Convergence-2025'', grant 1.7.01.4.
\\
The research activity of F.M. has been carried out in the framework of the
activities of the National Group of Mathematical Physics (GNFM, INdAM).
\\
{The authors are grateful to the anonymous referees for valuable suggestions which help us to improve the presentation of the results.}

\vspace{5mm}
\noindent
The authors declare that they have no conflict of interests.

\end{document}